\documentclass[12pt]{article}
\usepackage[utf8]{inputenc}
\usepackage{amssymb}
\usepackage{amsmath}
\usepackage{amsthm}
\usepackage[T1]{fontenc}
\usepackage{wasysym}
\usepackage{verbatim}
\usepackage{graphicx}
\usepackage{color}
\usepackage[margin=1.0in]{geometry}

\theoremstyle{plain}
\newtheorem{theorem}{Theorem}[section]

\theoremstyle{definition}
\newtheorem{definition}{Definition}[section]

\newcommand{\ignore}[1]{}

\newcommand{\hcm}[1][1]{\hspace*{#1 cm}}

\newcommand{\istrut}[2][0]{\rule[- #1 mm]{0mm}{#1 mm}\rule{0mm}{#2 mm}}

\newcommand{\paren}[1]{\left( #1 \right)}

\newcommand{\ceil}[1]{\lceil #1 \rceil}
\newcommand{\floor}[1]{\lfloor #1 \rfloor}

\newcommand{\DS}{\lambda}

\newcommand{\Kovari}{K\H{o}v\'{a}ri}
\newcommand{\Sos}{S\'{o}s}
\newcommand{\Turan}{Tur\'{a}n}
\newcommand{\Kollar}{Koll\'{a}r}
\newcommand{\Ronyai}{R\'{o}nyai}
\newcommand{\Szabo}{Szab\'{o}}

\begin{document}

\title{Lower Bounds on Davenport-Schinzel Sequences\\ via Rectangular Zarankiewicz Matrices}

\author{Julian Wellman\thanks{Work done as part of the Advanced Research course at Greenhills School, taught by Julie Smith}\\
Greenhills School\\
Ann Arbor, MI
\and
Seth Pettie\thanks{Supported by NSF grants CNS-1318294, CCF-1514383, and CCF-1637546.}\\
University of Michigan\\
Ann Arbor, MI}

\date{}
\maketitle

\begin{abstract}

An \emph{order-$s$ Davenport-Schinzel sequence} over an $n$-letter alphabet is one avoiding immediate repetitions and alternating subsequences with length $s+2$.  The main problem is to determine the maximum length of such a sequence, as a function of $n$ and $s$.  When $s$ is \emph{fixed} this problem has been settled (see Agarwal, Sharir, and Shor~\cite{ASS89},
Nivasch~\cite{Nivasch10} and Pettie~\cite{Pettie-DS-JACM}) but when $s$ is a function of $n$, very little is known about the extremal function $\DS(s,n)$ of such sequences.

In this paper we give a new recursive construction of Davenport-Schinzel sequences that is based on dense 0-1 matrices avoiding large all-1 submatrices (aka \emph{Zarankiewicz's Problem}.) In particular, we give a simple construction of $n^{2/t} \times n$ matrices containing $n^{1+1/t}$ 1s that avoid $t\times 2$ all-1 submatrices.

Our lower bounds on $\DS(s,n)$ exhibit three qualitatively different behaviors depending on the size of $s$ relative to $n$.  When $s \le \log\log n$ we show that $\DS(s,n)/n \ge 2^s$ grows exponentially with $s$. When $s = n^{o(1)}$ we show $\DS(s,n)/n \ge (\frac{s}{2\log\log_s n})^{\log\log_s n}$ grows faster than any polynomial in $s$.  Finally, when $s=\Omega(n^{1/t}(t-1)!)$, $\DS(s,n) = \Omega(n^2 s/(t-1)!)$ matches the trivial upper bound $O(n^2s)$ asymptotically, whenever $t$ is constant.
\end{abstract}

\section{Introduction}

In 1965 Davenport and Schinzel~\cite{DS65} introduced the problem of bounding the maximum length of a sequence on an alphabet of $n$ symbols that avoids any subsequence of the form $a\cdots b \cdots a \cdots b \cdots$ of length $s + 2$. We call any sequence $S$ which does not contain immediate repetitions and which does not contain an alternating subsequence of length $s + 2$ a Davenport-Schinzel (DS) sequence of order $s$. Let $|S|$ be the length of $S$, $\|S\|$ be the number of distinct symbols in $S$, and $DS(s,n)$ be the set of all Davenport-Schinzel sequences of order $s$ on $n$ symbols.  
We are interested in bounding the extremal function for
DS sequences. 
\[
\DS(s,n) = \max \{|S| \::\: S \in DS(s,n) \}
\]

The behavior of $\DS(s,n)$ is well understood when $s$ is fixed~\cite{HS86,ASS89,Nivasch10,Pettie-DS-JACM}, or when $s\ge n$~\cite{RoselleS71}.  However, very little is known when $s$ is a function of $n$ and $1 \ll s \ll n$.

\subsection{Fixed-order Davenport-Schinzel Sequences}

Most investigations of DS sequences has focused on the case of fixed $s$.  This is motivated by applications in computational geometry~\cite{AgarwalSharir95,SharirCKLPS86}, where DS sequences are used to bound the complexity of the lower envelope of $n$ univariate functions, each pair of which cross at most $s$ times, e.g., a set of $n$ degree-$s$ polynomials.  The following theorem synthesizes results of Davenport and Schinzel~\cite{DS65} ($s\in \{1,2\}$), Agarwal, Sharir, and Shor~\cite{ASS89} (sharp bounds for $s=4$, lower bounds for even $s\ge 6$), Nivasch~\cite{Nivasch10} (lower bounds for $s=3$, upper bounds for even $s\ge 6$), and Pettie~\cite{Pettie-DS-JACM} (upper bounds for all odd $s\ge 3$, lower bounds for $s=5$). 
Refer to Klazar~\cite{Klazar02} for a history of Davenport-Schinzel sequences
from 1965--2002, and Pettie~\cite{Pettie-DS-JACM,Pettie-GenDS11,Pettie15-SIDMA} for recent developments.

\begin{theorem}\label{thm:constant-s}
When $s$ is fixed, the asymptotic behavior of $\DS(s,n)$, as a function of $n$, is as follows.	
\[
\DS(s,n) = \left\{
\begin{array}{l@{\hcm}l@{\istrut[3]{0}}}
n									& s=1\\
2n-1								& s=2\\
2n\alpha(n) + O(n)					& s=3\\
\Theta(n2^{\alpha(n)})				& s=4\\
\Theta(n\alpha(n)2^{\alpha(n)})		& s=5\\
n\cdot 2^{(1 + o(1))\alpha^t(n)/t!}	& \mbox{for both even and odd } s\ge 6, \; t = \floor{\frac{s-2}{2}}.
\end{array}
\right.
\]
\end{theorem}

Here $\alpha(n)$ is the slowly growing inverse-Ackermann function.  Observe that if we regard $\alpha(n)$ as a constant, the dependence of $\DS(s,n)$ on $s$ is {\em doubly exponential}. 
This doubly exponential growth can be extended to non-constant $s$,
but the constructions of~\cite{ASS89,Nivasch10,Pettie-DS-JACM} only
work when $s=O(\alpha(n))$.  When $s = \Omega(\alpha(n))$ the existing lower bounds break down, but the upper bounds of~\cite{ASS89,Nivasch10,Pettie-DS-JACM} continue to give non-trivial upper bounds for $s=o(\log n)$. 
They imply, for example, 
that $\DS(s,n) = O(n(\log^{\star\star\cdots\star}(n))^{s-2}))$, for any fixed number of stars.\footnote{This result is not stated explicitly in~\cite{ASS89,Nivasch10,Pettie-DS-JACM}, but it is straightforward to cobble together, e.g., from Pettie~\cite[Lemma 3.1(2,4) and Recurrence 3.3]{Pettie-DS-JACM}.  The $\star$ operator is defined for any $f$ that is strictly decreasing on $\mathbb{N}\backslash\{0\}$.  By definition $f^\star(n)=\min\{i\;|\; f^{(i)}(n) \le 1\}$, where $f^{(i)}$ is the $i$-fold iteration of $f$.}

\subsection{Large-order Davenport-Schinzel Sequences}\label{sect:RoselleStanton}

A straightforward pigeonhole argument (see \cite[p. 3]{Klazar02}) gives the following upper bound on $\DS(s,n)$.
\begin{equation}\label{eqn:ub}
\DS(s,n) \leq \binom{n}{2}s + 1 
\end{equation}
For fixed $s$ this bound is off by nearly a factor $n$, but for \emph{fixed $n$} this bound is quite tight as a function of $s$.  In fact, Roselle and Stanton~\cite{RoselleS71} showed that for $s=\Omega(n)$, $\DS(s,n) = \Theta(n^2s)$, and that for fixed $n$, $\lim_{s\rightarrow\infty} \DS(s,n)/s = \binom{n}{2}$, i.e., Eqn.~(\ref{eqn:ub}) is sharp up to the leading constant $\binom{n}{2}$.
Let us give a brief description of Roselle and Stanton's construction.  
The sequence $RS(s,n)[a_1,a_2,\ldots,a_n]$ is a $DS(s,n)$ sequence constructed from the alphabet $\{a_1,\ldots,a_n\}$ in which the first occurrences of each symbol are in the order $a_1a_2\cdots a_n$.
If omitted, take the alphabet to be $[1,2,\ldots,n]$.
The construction is recursive, and bottoms out in one of two base cases, depending on whether $s>n$ or $s\le n$ initially.
\begin{align*}
RS(2,n) &= 121314\cdots 1(n-1)1n1\\
RS(s,2)	&= 1212\cdots \mbox{(length $s+1$)}
\intertext{When $s,n > 2$ we construct $RS(s,n)$ inductively.}
RS(s,n) &= \operatorname{Alt}(s,n)\cdot RS(s-1,n-1)[n,n-1,\ldots,2]\\
\mbox{where }\;
\operatorname{Alt}(s,n) &= \overbrace{1212\cdots 12}^{\ceil{\frac{s-2}{2}}\; 2\mathrm{s}} \overbrace{1313\cdots13}^{\ceil{\frac{s-2}{2}}\; 3\mathrm{s}}14\cdots \overbrace{1n1n\cdots 1n1}^{\ceil{\frac{s-2}{2}}\; n\mathrm{s}}
\end{align*}

In other words, with $\operatorname{Alt}(s,n)$ we introduce the maximum number of alternations between $1$ and each $k\in\{2,\ldots,n\}$, then ``retire'' the symbol $1$ and append a copy of $RS(s-1,n-1)$ on the alphabet $\{2,\ldots,n\}$.  Observe that it is crucial that the remaining alphabet be `reversed' in the recursive invocation of $RS(s-1,n-1)$.  In $\operatorname{Alt}(s,n)$ the symbols $2,3,\ldots,n$ appeared in this order, so to minimize the number of alternations the symbols in $RS(s-1,n-1)$ should make their first appearances in the order $n,n-1,\ldots,2$.  It is easily seen that $|\operatorname{Alt}(s,n)| = \Theta(sn)$ and $|RS(s,n)| = \Theta(\min\{n^2s,ns^2\})$, depending on whether $s>n$ or $s\le n$.  See~\cite{RoselleS71} for a careful analysis of the leading constant and lower order terms.

\subsection{Summary and New Results}

Suppose we fix $n$ at some very large value and let $s$ increase.  Theorem~\ref{thm:constant-s} (and a close inspection of the constructions of~\cite{ASS89,Nivasch10,Pettie-DS-JACM}) shows that $\DS(s,n)/n$ grows doubly exponentially with $s$, but only up to $s=O(\alpha(n))$.
For somewhat larger $s$ the best lower bounds on $\DS(s,n)/n$ are quadratic ($\Omega(s^2)$) \cite{DS65,RoselleS71} and best upper bounds exponential ($(\log^{\star\cdots\star}(n))^{s-2}$).  Eventually $s\ge n$ and $\DS(s,n)/s$ is known to be $\Theta(n^2)$, tending to $\binom{n}{2}$ in the limit~\cite{RoselleS71}.  Thus, when $1 \ll s \ll n$ we know very little about the true behavior of the extremal function $\DS(s,n)$.

In this paper we present a new construction of Davenport-Schinzel sequences that bridges the gap between the small-order ($s=O(\alpha(n))$) and large-order ($s=\Omega(n)$) regimes.  It exhibits three new \emph{qualitatively} different lower bounds on $\DS(s,n)/n$.
\begin{itemize}
\item When $s\le \log\log n$, $\DS(s,n)/n = \Omega(2^s)$ grows at least (singly) exponentially in $s$, which improves on~\cite{ASS89,Nivasch10,Pettie-DS-JACM} when $s\ge 2^{\Omega(\alpha(n))}$.

\item When $s > \log\log n$ we have $\DS(s,n)/n = \Omega((\frac{s}{2\log\log_s n})^{\log\log_s n})$.  For example, $\DS(\log n,n)/n > 2^{\Omega((\log\log n)^2)}$ is quasi-polylogarithmic in $n$.

\item Suppose that $s \ge n^{1/t}(t-1)!$ for an integer $t$.  In this case we obtain asymptotically sharp lower bounds on $\DS(s,n) = \Omega(n^2 s/(t-1)!)$ whenever $t$ is constant.
\end{itemize}

\subsection{Overview of the Paper} 

In Section~\ref{sect:warmup} we give a simple construction showing that $\DS(s,n)/n = \Omega(2^s)$, for $s$ up to $\log\log n$. In Section~\ref{sect:Zarankiewicz} we construct an $n^{2/t} \times n$ Zarankiewicz matrix with $n^{1 + 1/t}$ 1s which avoids $t \times 2$ all-1 submatrices. 
Zarankiewicz matrices are used in Section~\ref{sect:LargeConstruction} to construct Davenport-Schinzel sequences of length $\Omega(n^2s/(t - 1)!)$ when $s \geq n^{1/t}(t-1)!$. 
The space where $\log \log n \ll s \ll n^{o(1)}$ is addressed in Section~\ref{sect:MediumConstruction}.
We conclude with some remarks and open problems in Section~\ref{sect:conclusion}.

\section{A Simple Construction for Small Orders}\label{sect:warmup}

In this section we present a simple construction for the special case $s = \log\log n+2$, which can easily be scaled down to the case when $s\le \log\log n+1$. 
The sequence $S(k)$ is an order-$s(k)$ DS sequence over an $n(k)$-letter alphabet in which each symbol occurs $\mu(k)$ times.  We will construct $S(k+1)$ inductively from $S(k)$ and thereby obtain recursive definitions for $n(k+1),s(k+1),\mu(k+1)$.  Let $S(k)[a_1,\ldots,a_{n(k)}]$ denote a copy of $S(k)$ in which the letters $a_1,\ldots,a_{n(k)}$ make their first appearance in that order, and let $\overline{S}$ be the reversal of $S$.  If left unspecified, the alphabet is $[1,\ldots,n(k)]$.

In the base case $k=0$ we let $S(0) = 12$.  Thus,
\begin{align*}
n(0) &= 2, & \mu(0) &= 1,  & s(0) &= 1.
\end{align*}

Now we construct $S(k + 1)$ from $S(k)$. Arrange $n(k)^2$ distinct symbols in an $n(k) \times n(k)$ matrix. Let $C_i$ (and $R_i$) be the sequences of symbols in column $i$ (and row $i$), $1 \leq i \leq n(k)$, listed in increasing order of row index (and column index).  The sequence $S(k+1)$ is constructed as follows:

\[
S(k+1) = \overline{S(k)[C_1]}\:\overline{S(k)[C_2]}\cdots \overline{S(k)[C_{n(k)}]}
S(k)[\overline{R_1}]\, S(k)[\overline{R_2}] \cdots S(k)[\overline{R_{n(k)}}]
\]

It follows that $S(k+1)$ has the following parameters.
\begin{align*}
n(k+1) &= n(k)^2		& \mu(k+1) &= 2\mu(k)	& s(k+1)=\max\{3,s(k)+1\}
\end{align*}
The expression for $n(k+1)$ is by construction and the expression for $\mu(k+1)$ follows from the fact that each symbol appears in one row and one column.  The claim that $s(k+1) = \max\{3,s(k)+1\}$ requires a more careful argument.  Consider two symbols $a,b$ at positions $(i,j)$ and $(i',j')$ in the $n(k)\times n(k)$ symbol matrix.  If $i\neq i'$ and $j\neq j'$ then we may see the subsequence $abab$ in $S(k+1)$, but never $ababa$.  Suppose that $i=i'$ and $j<j'$.  In the first half of $S(k+1)$, all $a$s (in $\overline{S(k)[C_j]}$) precede all $b$s (in $\overline{S(k)[C_{j'}]}$) and in the second half of $S(k+1)$, all occurrences of $a$ and $b$ appear in $S(k)[\overline{R_i}]$.  Moreover, because $a$ precedes $b$ in $R_i$, the first occurrence of $b$ precedes the first occurrence of $a$ in $S(k)[\overline{R_i}]$.
Symmetric observations hold when $i<i'$ and $j=j'$.
Thus, for any two symbols $a,b$, either $ababa$ does not appear in $S(k+1)$ or $S(k+1)$ introduces one more alternation than $S(k)$.  
We conclude that $s(k+1) = \max\{3,s(k)+1\}$.

By induction on $k$, We have the following closed form bounds on the parameters of $S(k)$.
\begin{align*}
n(k) &= 2^{2^k} \\
s(k) &= k + 2\\
\mu(k) &= 2^k
\end{align*}

As constructed $S(k+1)$ contains immediate repetitions: the last symbol of $\overline{S(k)[C_{n(k)}]}$ is identical to the first symbol of $S(k)[\overline{R_1}]$.  In order to make $S(k+1)$ a proper order-$s(k+1)$ DS sequence we must remove one of these copies, and apply the procedure recursively to each copy of $S(k)$.  The fraction of occurrences removed is slightly more than $1/8$.\footnote{It is dominated by the occurrences removed in copies of $S(1)$, which has length $8$ originally and length 7 afterward.}

\begin{theorem}
For any $s\le \log\log n+2$, $\lambda(s,n) = \Omega(n \cdot 2^s).$
\end{theorem}
\begin{proof}
Partition the alphabet $[n]$ into subsets of size $n' = 2^{2^{s-2}}$ and concatenate $\floor{n/n'}$ copies of $S(s-2)$, one on each part of the alphabet.  Each part has length $\Omega(n'2^{s-2})$, so the whole sequence has length $\Omega(n2^{s-2})$.
\end{proof}

In the case of $s = \log \log n+2$, we can get a sequence of length $\Omega(n \log n)$, which is not known from prior constructions.  The longest sequences that can be generated using~\cite{Nivasch10,Pettie-DS-JACM,Pettie15-SIDMA}
have length $O(n2^{2^{\alpha(n)}})$.

\section{Rectangular Zarankiewicz Matrices}\label{sect:Zarankiewicz}

The construction of the previous section is limited by the fact that each letter of $S(k+1)$ appears in only two copies of $S(k)$ (corresponding to the letter's row and column).  In order to bound $s(k+1) \le s(k)+1$, it was crucial that each pair of symbols appeared in only one \emph{common} copy of $S(k)$.  In general, one could imagine generalized constructions of $S(k+1)$ over an $n$-letter alphabet that are formed by concatenating $m$ copies of $S(k)$, each over a subset of the alphabet, with the property that two symbols do not appear in \emph{too many} common subsets.  Designing such a system of subsets is an instance of Zarankiewicz's problem.

\begin{definition} (Zarankiewicz's Problem)
Define $z(m,n; s,t)$ to be the maximum number of 1s in an $m\times n$ 0-1 matrix that contains no all-1 $s\times t$ submatrix.  Define $z(n,t)$ to be short for $z(n,n; t,t)$.
\end{definition}

The \Kovari, \Sos, and \Turan\ theorem~\cite{KovariST54}, explicitly proven in~\cite{Hylten-Cavallius}, gives the following general upper bound on $z(m,n;s,t)$.
\[
z(m,n; s,t) \leq (s-1)^{1/t}(n-t+1)m^{1-1/t}+(t-1)m
\]
It is generally believed that the \Kovari-\Sos-\Turan{} 
upper bound on $z(n,t) = O(n^{2-1/t})$ is asymptotically sharp, but
this has only been established for $t\in\{2,3\}$~\cite{Brown66}.  
\Kollar, \Ronyai, and \Szabo~\cite{KollarRS96} gave sharp bounds on 
$z(n,n,t!+1,t) = \Omega(n^{2-1/t})$, where the forbidden submatrix is highly skewed.  In this paper we need bounds on Zarankiewicz's problem in which both the $m\times n$ matrix and forbidden pattern are rectangular.  The following theorem may be folklore in some quarters; nonetheless, it is not mentioned in a recent survey~\cite{FurediSimonovits}. The only existing construction avoiding $t \times 2$ 
all-1 submatrices is tailored to square matrices~\cite{Furedi96b}.

\begin{theorem}\label{thm:Z}
For any fixed integer $t\ge 2$ and large enough $n$,
\[
z(n^{2/t}, n, t, 2) = \Theta(n^{1+1/t}).
\]
\end{theorem}

\begin{proof}
Let $q$ a prime power and $\mathbb{F}$ be the Galois field of order $q$.
We will show that $z(q^2, q^{t}, t, 2) \ge q^{t+1}$.  By~\cite{KovariST54} this bound is asymptotically sharp.  It is straightforward to extend this to any $n$ (not of the form $q^t$) with a constant factor loss.

We will construct a matrix $A \in \{0,1\}^{q^2\times q^t}$ as follows.  The columns of $A$ are indexed by all degree-$(t-1)$ polynomials over $\mathbb{F}$.  A polynomial $f_{\mathbf{c}}$ is identified with its coefficient vector $\mathbf{c} = (c_0,c_1,\ldots,c_{t-1}) \in \mathbb{F}^t$, where
\[
f_{\mathbf{c}}(x) = \sum_{i=0}^{t-1} c_ix^i.
\]
The rows of $A$ are indexed by \emph{evaluations} $(x,v)\in\mathbb{F}^2$.
The matrix $A$ is generated by putting a 1 wherever we see a correct evaluation:
\[
A((x,v),\mathbf{c}) = \left\{\begin{array}{ll}
1	& \;\;\mbox{ if $f_{\mathbf{c}}(x) = v$}\\
0	& \;\;\mbox{ otherwise.}
\end{array}
\right.
\]

Suppose $A$ actually contains a $t\times 2$ all-1 submatrix defined by rows 
$\{(x_i,v_i)\}_{i\in[0,t)}$ and columns $\mathbf{c},\mathbf{c}'$.
Clearly $x_0,\ldots,x_{t-1}$ are distinct field elements.
It follows from the definition of $A$ that 
$f_{\mathbf{c}}(x_i) - f_{\mathbf{c}'}(x_i) = 0$ for each $i\in [0,t)$.
However $(f_{\mathbf{c}} - f_{\mathbf{c}'})(x) = \sum_{i=0}^{t-1} (c_i-c_i')x^i$ is a degree-$(t-1)$ polynomial over $\mathbb{F}$ and therefore has at most $t-1$ roots.  It is impossible for $f_{\mathbf{c}} - f_{\mathbf{c}'}$ to have $t$ distinct roots, namely $x_0,\ldots,x_{t-1}$.

Each row $(x,v)$ of $A$ has precisely $q^{t-1}$ 1s, since for any partial coefficient vector $(c_1,\ldots,c_{t-1})$, there is some $c_0$ for which
$f_{(c_0,\ldots,c_{t-1})}(x)=v$.  Similarly, each column $\mathbf{c}$ of $A$ has precisely $q$ 1s since there is one value $v$ for which $f_{\mathbf{c}}(x)=v$.  Thus, $A$ contains precisely $q^{t+1}$ 1s.
\end{proof}

\section{Polynomial Order Davenport-Schinzel Sequences}\label{sect:LargeConstruction}

Let $q$ be a prime power and $\hat{s}\ge q$ be a parameter.
For each integer $t\ge 1$ we will construct an order-$O(\hat{s}(t-1)!)$ 
sequence $S_t(\hat{s},q)$ over an alphabet of size $q^t$ 
with length $\Omega(q^{2t}\hat{s})$.  
Phrased in terms of $n=q^t$ and $s=O(\hat{s}(t-1)!)$, this shows that $\DS(s,n) = \Omega(n^2s/(t-1)!)$.
The construction is inductive.  In the base case $t=1$ we revert to Roselle and Stanton's construction. (See Section~\ref{sect:RoselleStanton}.)
\[
S_1(\hat{s},q) = RS(\hat{s},q^t).
\]
Now suppose that $t\ge 2$.  Let $A$ be the $q^2\times q^t$ 0--1 matrix from Theorem~\ref{thm:Z}.  Each column of $A$ is identified with a symbol in the alphabet of $S_t(\hat{s},q)$ and each row is 
identified with a subset of its alphabet.
In particular, let 
$C_i$, $i\in[1,q^2]$, be the list of columns (symbols) in which $A(i,\star)=1$. 
We form $S_t(\hat{s},q)$ as follows:
\[
S_t(\hat{s},q) = S_{t-1}(\hat{s},q)[C_1]\cdot S_{t-1}(\hat{s},q)[C_2]\cdots S_{t-1}(\hat{s},q)[C_{q^2}],
\]
where $S_{t-1}(\hat{s},q)[X]$ is a copy of $S_{t-1}(\hat{s},q)$ over the alphabet $X$.  According to the proof of Theorem~\ref{thm:Z} $|C_i|=q^{t-1}$, so the alphabets have the requisite cardinality.  By construction we have
\begin{align*}
|S_t(\hat{s},q)| &= q^2\cdot |S_{t-1}(\hat{s},q)|\\
			&= q^2 \cdot \Omega(q^{2(t-1)}\hat{s})	& \mbox{inductive hypothesis}\\
                		&= \Omega(q^{2t}\hat{s}).
\end{align*}
Let $s_t=s(t,\hat{s},q)$ be the length of the longest alternating subsequence in $S_t(\hat{s},q)$, which would make it an order-$(s_t-1)$ DS sequence.  We want to bound $s_t$ in terms of $s_{t-1}$.  Pick two arbitrary symbols $a,b$.  Because $A$ avoids all-1 $t\times 2$ submatrices, $a$ and $b$ appear in up to $t-1$ common subsets among $\{C_i\}$ and therefore at least $q-(t-1)$ subsets in which the other does not appear.  Each subset of the first type contributes $s_{t-1}$ alternations between $a$ and $b$ and each subset of the second type contributes  1, in the worst case where they happen to be interleaved.  Thus, we have the following recursive expression for $s_t$.
\begin{align*}
s_1 &= \hat{s}+1		& \mbox{(because $RS(\hat{s},q^t)$ is an order-$\hat{s}$ DS sequence)}\\
s_t &\le (t-1)s_{t-1} + 2(q-t+1)
\end{align*}
Since $\hat{s}\ge q$, $s_t = O((t-1)!\hat{s})$.

\begin{theorem}\label{thm:DSlarge}
When $s = \Omega(n^{1/t}(t-1)!)$, 
$\DS(s,n)$ is $\Omega(n^2 s/(t-1)!)$ and $O(n^2s)$.
\end{theorem}

\section{Medium Order Davenport-Schinzel Sequences}\label{sect:MediumConstruction}

The construction of Theorem~\ref{thm:DSlarge} is asymptotically sharp when $t$ is constant (and $s$ polynomial in $n$), but becomes trivial when $t=\omega(\log n/\log\log n)$.  In this section we design a simpler construction that works well when $\log\log n < s = n^{o(1)}$. 

The construction is parameterized by a prime power $q$ and parameter $\hat{s}\leq q$.  The sequence $S_t(\hat{s},q)$ will be a sequence over an alphabet of size $q^{2^t}$.
In the base case $t=0$ we have
\[
S_0(\hat{s},q) = RS(\hat{s},q)
\]
so $|S_0(\hat{s},q)| = \Theta(q\hat{s}^2)$.  When $t\ge 1$ we build $S_t(\hat{s},q)$ using a truncated version of the Zarankiewicz matrix from Theorem~\ref{thm:Z}.
Let $\hat{q}=q^{2^{t-1}}$ and $A$ be the $\hat{q}^2\times \hat{q}^2$ 0-1 matrix avoiding $2\times 2$ all-1 submatrices.  
Let $A'$ consist of the first $\hat{q}\hat{s}$ rows of $A$; 
i.e., each row of $A'$ has $\hat{q}$ 1s and each column of $A'$ has $\hat{s}$ 1s. This particular matrix could have been constructed using mutually orthogonal Latin squares, still generated based on finite fields as in~\cite{Furedi96b} or~\cite{Bose38}.
As in Section~\ref{sect:LargeConstruction} we identify the columns with symbols and the rows with sequences of symbols $C_1,\ldots,C_{\hat{q}\hat{s}}$. The sequence $S_t(\hat{s},q)$ is formed as follows:
\[
S_t(\hat{s},q) = S_{t-1}(\hat{s},q)[C_1]\cdot S_{t-1}(\hat{s},q)[C_2] \cdots S_{t-1}(\hat{s},q)[C_{\hat{q},\hat{s}}]
\]
Assuming inductively that $|S_{t-1}(\hat{s},q)| = \Omega(q^{2^{t-1}}\hat{s}^{t+1})$,
we have
\begin{align*}
|S_{t}(\hat{s},q)| &= q^{2^{t-1}}\hat{s} |S_{t-1}(\hat{s},q)|\\
				&= q^{2^{t-1}}\hat{s}\cdot\Theta(q^{2^{t-1}}\hat{s}^{t+1})\\
                &= \Theta(q^{2^t}\hat{s}^{t+2})
\end{align*}
Each symbol appears in exactly $\hat{s}$ distinct sequences among $\{C_i\}$
and any two symbols appear in at most one common sequence among $\{C_i\}$.  Thus, if $s_t$ is the length of the longest alternating sequence in $S_t(\hat{s},q)$, we have
\begin{align*}
s_0 &= \hat{s}+1 	& \mbox{($RS(\hat{s},\star)$ is an order-$\hat{s}$ DS sequence)}\\
s_t &= s_{t-1} + 2(\hat{s}-1)
\end{align*}
Clearly $s_t = (2t+1)(\hat{s}-1) + 2$.
In terms of the alphabet size $n=q^{2^t}$, $t = \log\log_q n$.
In terms of $s=s_t$ and $n$, the length of $S_t(\hat{s},q)$ is
\[
\Theta(n\hat{s}^{t+2}) 
= \Omega\paren{n\paren{\frac{s}{2\log\log_q n+1}}^{\log\log_q n+2}}
= \Omega\paren{n\paren{\frac{s}{2\log\log_s n}}^{\log\log_s n}}.
\]

\begin{theorem}
For any $s=\Omega(\log\log n)$, $\DS(s,n) = \Omega(n(\frac{s}{2\log\log_s n+1})^{\log\log_s n + 1})$. 
For example, $\DS(\log n, n)/n = 2^{\Omega((\log\log n)^2)}$
is quasi-polylogarithmic in $n$.
\end{theorem}

\section{Conclusion and Open Problems}\label{sect:conclusion}

We have attained asymptotically tight bounds on $\DS(s,n)$ when $s = n^{\epsilon}$. Specifically, the trivial upper bound $\DS(s,n) = O(n^2 \cdot s)$ can be achieved asymptotically, with the leading constant depending on $\epsilon$.  Even when $s = n$ the true leading constant of $\DS(n,n)$ is only known approximately; it is in the interval [1/3,1/2]~\cite{RoselleS71,Klazar02}.  Several interesting open problems remain, among them:

\begin{itemize}
\item Our lower bounds on $\DS(s,n)$ when $1\ll s \ll n^{o(1)}$ are quite far from the best upper bounds in this range~\cite{Nivasch10,Pettie-DS-JACM}.  It is still consistent with all published results that $\DS(s,n)/n$ grows (at least) exponentially in $s$ for all $s\le \log n$, and that $\DS(s,n)=\Theta(n^2s)$ for $s\ge \log n$.

\item Our constructions are not very robust to slight variants in the definition of the extremal function $\DS(s,n)$.   For example, if we insist that the sequence be 3-sparse (every three consecutive symbols must be distinct) rather than 2-sparse (merely avoiding immediate repetitions), the Roselle-Stanton construction no longer works and we cannot claim that when $s>n^\epsilon$, $\DS(s,n) = \Omega(n^{2}s)$ is witnessed by some $3$-sparse sequence.  This is in sharp contrast to the fixed-$s$ world~\cite{Nivasch10,Pettie-DS-JACM}, which are highly robust to different notions of sparseness.

\item A popular way to constrain DS sequences is to specify the number blocks~\cite{ASS89,Nivasch10,Pettie-DS-JACM}. A \emph{block} is a sequence of distinct symbols.  Let $\DS(s,n,m)$ be the length of an order-$s$ DS sequence over an $n$-letter alphabet that is partitioned into $m$ blocks.  In the fixed-$s$ world~\cite{Nivasch10,Pettie-DS-JACM}, 
$\DS(s,n)$ is roughly $\DS(s,n,n)$; see, e.g., \cite[Lemma 3.1]{Pettie-DS-JACM}.  
Our constructions for $s$ in the ``small'' and ``medium'' range do give non-trivial bounds on $\DS(s,n,n)$, but say nothing interesting when $s=n^\epsilon$.  Bounding $\DS(s,n,n)$ is essentially identical~\cite{Pettie-FH11} to bounding the number of 1s in an $n\times n$ 0-1 matrix avoiding $2\times (s+1)$ alternating submatrices of the following form.
\[
\left(
\begin{array}{cccccc}
	&	1	&		& \cdots &	1	&\\
1	&		&	1	&		& 		&	1
\end{array}
\right)
\]
Clearly the extremal function $\DS(s,n,n)$ tends to $n^2$ 
as $s\rightarrow n$, but we know very little about the rate of convergence.  For example, how large must $s$ be in order for
$\DS(s,n,n) = \Omega(n^{2-o(1)})$?
\end{itemize}


\end{document}